\documentclass[12pt]{amsart}

\usepackage{amsmath, amssymb, amsfonts}
\usepackage{latexsym}
\usepackage{color}
\usepackage{tikz}
\usepackage{cases}
\usepackage{systeme}
\usepackage{breqn}
\usepackage{graphics}
\usepackage{bbm}
\usepackage[utf8]{inputenc} 
\usepackage[T1]{fontenc}
\usepackage{xcolor}
\usepackage{float}
\usepackage{srcltx}
\usepackage{enumerate}
\usepackage{esint}
\usepackage[shortlabels]{enumitem}
\usepackage{hyperref}
\usepackage[capitalise]{cleveref}

\makeatletter
\@namedef{subjclassname@2020}{%
  \textup{2020} Mathematics Subject Classification}
\makeatother

\newtheorem{theorem}{Theorem}[section]
\newtheorem{lemma}[theorem]{Lemma}
\newtheorem{corollary}[theorem]{Corollary}
\newtheorem{proposition}[theorem]{Proposition}

\newtheorem{conjecture}[theorem]{Conjecture}

\theoremstyle{definition}
\newtheorem{definition}[theorem]{Definition}

\theoremstyle{remark}
\newtheorem{remark}[theorem]{Remark}
\DeclareMathOperator{\Tr}{Tr}

\numberwithin{equation}{section}
\numberwithin{equation}{section}

\def \bk {\color{black}}
\def \rd {\color{black}}
\def \bll {\color{black}}

\def \bl {\color{black}}

\def \a {\alpha}
\def \b {\beta}

\begin{document}

\title{On Weil Sums, Conjectures of Helleseth, and Niho Exponents}

\author{Liem Nguyen}
\address{Department of Mathematics\\
Louisiana State University\\
Baton Rouge, Louisiana, 70808}
\email{lngu145@lsu.edu}

\date{\today}
\thanks{The author is partially supported by NSF DMS \#$1602047$.}

\maketitle

\begin{abstract}
Let $F$ be a  finite field,  $\mu$ be a fixed additive character and $s$ be an integer coprime to $|F^\times|$. For any $a\in F$, the corresponding Weil sum is defined to be $W_{F,s}(a)=\displaystyle\sum_{x \in F} \mu(x^s-ax)$. The Weil spectrum counts distinct values of the Weil sum as $a$ runs through the invertible elements in the finite field. The value of these sums and the size of the Weil spectrum are of particular interest, as they link problems in coding and information theory to other areas  of  math  such  as  number  theory  and  arithmetic  geometry. \bll In the setting of Niho exponents, we examine the Weil sum, its bounds and its spectrum. As a consequence, we give a new proof to the Vanishing Conjecture of Helleseth ($1971$) on the presence of zero in the Weil spectrum in the case of Niho exponents. \bk We also state a conjecture for when the Weil spectrum contains at least five elements, and prove it for a certain class of Weil sums.
\end{abstract}


\section{Introduction}
\label{introduction}
\subsection{\bl The Weil sum} \bk

Let $F$ be a finite field of characteristic $p$ and size $q=p^n$. Let $\mu: F \to \mathbb{C}$ be the canonical additive character. Recall the canonical additive character $\mu (x)= \zeta_p^{\Tr_{F/\mathbb{F}_p}(x)}$, where $\zeta_p=e^{2\pi i/p}$ is a $p$th root of unity and $\Tr_{F/\mathbb{F}_p}(x)$ is the absolute trace function from $F \to \mathbb{F}_p$. If $L$ is an extension of $F$, i.e $|L|=q^m$ for some nonnegative integer $m$, then \bl we take \bk $\mu(x)=\zeta_p^{\Tr_{F/\mathbb{F}_p}(\Tr_{L/F}(x))}$ where $\Tr_{L/F}(x)$ is the trace function from $L \to F$.  

Let $x \mapsto x^s$ be a power mapping from $F$ to $F$ where $s$ is a fixed positive integer. For a point $a \in F$, we define the Weil sum as
\[
W_{F,s}(a)=\sum_{x \in F}\mu(x^s-ax).
\]
If $\gcd(s,q-1)=1$, then $s$ is said to be an \emph{invertible exponent} \bl over $F$. \bk Furthermore, 
\begin{align*}\label{W(0)}
W_{F,s}(0)=\sum_{x \in F} \mu(x^s)= \sum_{x \in F} \mu(x)=0,
\end{align*}
since the map $x \mapsto x^s$ permutes the elements of $F$.

Since $W_{F,s}(a)$ is a sum of roots of unity, $W_{F,s}(a)$ is an algebraic integer. The following theorem by Tor Helleseth states the necessary and sufficient conditions for $W_{F,s}(a)$ to be a rational integer for every $a\in F$ \cite{tor}.

\begin{theorem}[Helleseth]
$W_{F,s}(a) \in \mathbb{Z}$ for all $a \in F^\times$ if and only if $s \equiv 1 \pmod{p-1}$. 
\end{theorem}
Now, $s$ is said to be \emph{singular} if there is an $a \in F^\times$ such that $W_{F,s}(a)=0$.

\bl

The Weil sum is of interest to us as it relates many problems in coding and information theory to other areas such as number theory and arithmetic geometry. Properties of the Weil sum including its values, number of values over the finite field, and its bounds are still not well understood. These aspects are certainly of interest from a purely number-theoretic standpoint. On the other hand, in sequence design in information theory, the cross-correlation function between two $p$-ary maximal linear recursive sequences measures how similar they are, and can be realized as a character sum, specifically as the Weil sum plus $(-1)$. One important criterion that makes such sequences useful in remote sensing and communications is that they should have low cross-correlation (see \cite{tor, Niho1972MultiValuedCF, Tormaster, sarwate, games, sawarteconjecture, Calderbank3value, WelchConjecture, WelchNihoConjecture}).

In coding theory, \rd a \bk $p$-ary m-sequence (by restricting to a finite subsequence) can be regarded as a codeword in \rd a \bk $p$-ary cyclic code. We can define the Hamming weight of a codeword to measure the number of substitutions to change one string to another. We can relate this weight to the cross-correlation function, which in turn relates to the Weil sum. We are interested in the distinct values of these weights, and the number of the weights, and hence the corresponding properties of the Weil sum (see \cite{mcguire, charpin, feng} for more details).

For a good summary of various research aspects of the Weil sum, as well as their applications, see \cite[Appendix]{katz} and \cite{Katz2018WeilSO} by Daniel Katz.

\bk 

\subsection{\bl Properties and Helleseth Conjectures} \bk

It is natural to wonder what kind of value one would get from the Weil sum. We have seen that the $W_{F,s}(a)$ is always $0$ at $a=0$, and interestingly, this presence of zero value is not known for nonzero elements $a$. This prompted Tor Helleseth to propose the following conjecture \cite{Tormaster,tor} in $1971$.   

\begin{conjecture}[Helleseth Vanishing Conjecture]
If $q=|F| > 2$ and $s$ is an invertible exponent over $F$ such that
$s \equiv 1 \pmod{p-1}$, then $s$ is singular.
\end{conjecture}

Now, if we put some restrictions on the exponent $s$, some partial results on the Vanishing Conjecture can be obtained. For the finite field $L$ of order $q=p^{2n}$, an exponent $s$ is called a \emph{Niho exponent} if $s$ is not a power of $p \bl \pmod{p^{2n}-1}$ \bk and $s \equiv p^j \pmod{p^n-1}$. If $j=0$, then such exponent is called a \emph{normalized Niho exponent}. Niho exponents \bl were \bk first introduced by Yoji Niho in $1972$ in his PhD thesis on the cross-correlation function between an $m$-sequence and its $d$-decimation \cite{Niho1972MultiValuedCF}. Since then further research has been done using Niho exponents, and it has resulted in various applications in coding theory, sequence design and cryptography \cite{Nihoexponent}. Moreover, the Helleseth Vanishing Conjecture was proved for Niho exponents for a field of characteristic $2$ in \cite{charpin} \bll and this was generalized for Niho exponents in all characteristics in \cite[Proof of Theorem $9$]{helleseth_lahtonen_rosendahl}. 

\bk

One useful fact about Weil sums with Niho exponents is that we can replace them with normalized Niho exponents due to the following result discussed in Aubry, Katz and Langevin paper \cite{cyclotomy}.

\begin{lemma}\cite[Lemma $3.2$]{cyclotomy}\label{powerlemma}
Let $F$ be a finite field of characteristic $p$ and $s$ be an invertible exponent over $F$.
Then $W_{F,s}(a)=W_{F,p^j s}(a)$ for any $a \in F$ and $j \in \mathbb{Z}$.
\end{lemma}

The proof of this lemma relies on the fact that $x^{p^js}$ and $x^s$ are Galois conjugates and thus have the same trace.

The next questions of interest would be how many distinct values $W_{F,s}(a)$ takes as $a$ ranges over $F$, and what they are. We define the \emph{Weil spectrum} for some fixed $s$ to be the set $\{W_{F,s}(a)\mid a\in F^\times\}$, and say that it is \emph{r-valued} if $|\{W_{F,s}(a)\mid a\in F^\times\}|=r$. 

If $s$ is a power of $p$ modulo $(q-1)$, we call $s$ \emph{degenerate}. For \rd a \bk degenerate power $s$ we know that $W_{F,s}(a)$ takes two values via a theorem by Helleseth \cite{tor}.

\begin{theorem}[Helleseth \cite{tor}] \label{2values}
If $s$ is degenerate, $W_{F,s}(a)$ is two-valued over $F$ where
\[
W_{F,s}(a)=
\begin{cases}
q & \text{if $a=1$},\\
0 & \text{otherwise}
\end{cases}
\]
If $s$ is nondegenerate, then $W_{F,s}(a)$ takes at least three values over $F^\times$.
\end{theorem}
So, when exactly is the Weil spectrum three-valued? In the same paper that Helleseth proposed the Vanishing Conjecture in $1971$, he also gave a criteria for when this three-valued property is never met \cite{Tormaster,tor}.
\begin{conjecture}[Helleseth Three-Valued Conjecture, 1971]
Let $F$ be a finite field of characteristic $p$. If $[F:\mathbb{F}_p]$ is a power of $2$, then for any invertible exponent $s$, the spectrum of the Weil sum $W_{F,s}(a)$ is not three-valued.
\end{conjecture}
More progress has been made towards this conjecture in comparison to the Vanishing Conjecture, using various approaches from coding theory, cryptography and number theory \cite{Calderbank3value, WelchConjecture, WelchNihoConjecture, mcguire,charpin,feng,katz,cyclotomy,divisibility,katz_langevin}. Currently, only ten families of three-valued Weil \rd spectra \bk are known \cite[Table 1]{cyclotomy}, and these are conjectured to be the only ones that occur. \bl The cases for characteristic $p=2$ and $p=3$ in the Three-Valued Conjecture were proven by Daniel Katz in \cite{katz} and in \cite{divisibility}, respectively \bk. Special families of the three-valued Weil sum for all characteristics $p$ are also addressed via the Welch Conjecture and the Niho Conjecture. Canteaut, Charpin, and Dobbertin gave a proof to the Welch Conjecture in \cite{WelchConjecture} and Hollmann and Xiang proved both the Welch and Niho Conjectures in \cite{WelchNihoConjecture}. More recently, a family of three-valued Weil \bl spectra \bk has been proved \cite{threevaluefamily}. 

The organization of this paper is as follows. The first part discusses how a Weil sum can be viewed as an inner product of characters over a finite field. This observation leads to a relation which is part of a method called power moments. In general, the method of power moments studies the summation $\displaystyle\sum W^k_{F,s}(a)$ for any positive integer $k$. The summation can be taken over a finite field, \bll and we look at a case where the summation is taken over different orbits in the field via multiplication action by a subfield. \bk This method is useful in the studies of distribution/averaging behavior and the divisibility of \rd the \bk Weil sum \cite{tor,divisibility}. 

Using the discussion from the first part we give an alternative proof of the Vanishing Conjecture for the case of a Niho exponent $s$ to \cite{helleseth_lahtonen_rosendahl}. The discussion then continues with obtaining bounds for the Weil sum in this setting.

\begin{theorem}\label{vanishingp^2n}
Let $L$  be a finite field where $q=p^{2n}$ for some odd prime $p$ and positive integer $n$. Suppose that $s$ is an invertible Niho exponent \rd over $L$. \bk Then $s$ is singular.
\end{theorem}

The last part of the paper focuses on computing the Weil sum for special values $a$ in the finite field and the Weil spectrum for the case of Niho \rd exponents. \bk Based on numerical evidence, we propose a conjecture for the five-valued behavior.

\bl
\begin{conjecture}\label{5valuesconjecture}
Let $L$ be a quadratic extension of a finite field $F$ of order $p^n$, where $p$ is an odd prime. Let $s=1+k(p^n-1)$ be an invertible Niho exponent over $L$, $d_1=\gcd(k, p^n+1)$, and $d_2=\gcd(k-1,p^n+1)$. 
If either
\begin{itemize}
    \item[(i)] $d_1+d_2 \geq 5$, or
    \item[(ii)] $d_1+d_2=3$ and $p^n \equiv 11 \pmod{12}$,
\end{itemize}
satisfies, then the Weil spectrum over $L$ is at least five-valued.

Moreover, in case (i), the five values are $\{0, -p^n, p^n, 2\alpha p^n, (2\beta+1)p^n\}$ where $\alpha, \beta \geq 1$ are integers. In case (ii), at least four values are $\{0,-p^n,p^n,2p^n\}$.
\end{conjecture}

\bk

A special case of the condition $d_1+d_2 \geq 5$ in \cref{5valuesconjecture} is $p^n \equiv 2 \pmod{3}$. Hence, we can restate the conjecture with simpler assumptions as follows.

\begin{conjecture}\label{5valuesconjecturep^n2mod3}
Let \bl p be an odd prime, \bk $L$ be a quadratic extension of a finite field $F$ of order $p^n$, and $s=1+k(p^n-1)$ be an invertible Niho exponent over $L$. If $p^n \equiv 2 \pmod{3}$, then the Weil spectrum has at least five values of the form $\{0, -p^n, p^n, 2\alpha p^n, (2\beta+1)p^n\}$ for integers $\alpha,\beta \geq 1$.
\end{conjecture}

\begin{remark}\label{remarkp^n2mod3} 
Since s is an invertible exponent over $L$, $\gcd(s,p^{2n}-1)=1$. Hence, if $p^n \equiv 2 \pmod{3}$, then $s \equiv 1$ or $2 \pmod{3}$. Thus, $k \equiv 0 \pmod{3}$ and $(k-1) \equiv 2 \pmod{3}$, or $k \equiv 1 \pmod{3}$ and $(k-1) \equiv 0 \pmod{3}$. Moreover, $p^n+1$ is divisible by $2$ and $3$. Therefore either $d_1$ or $d_2$ in \rd \cref{5valuesconjecture} \bk is divisible by $3$. The same conclusion can be made for the divisibility of either $d_1$ or $d_2$ by $2$. Hence, $d_1+d_2 \geq 5$.
\end{remark}

\bl 

We end the paper by showing case (i) of \cref{5valuesconjecture} holds true for sufficiently large primes. Finally, we then proved case (ii).

\begin{theorem}\label{theoremM_i}
Let $L$ be a quadratic extension of a finite field $F$ of order $p^n$, where $p$ is an odd prime and $n \geq 2$ is an integer. Let $k \geq 2$ be an integer such that $k <\displaystyle\frac{p}{2}+1$, and $s=1+k(p^n-1)$ \rd be \bk an invertible Niho exponent over $L$.
Let $d_1=\gcd(k, p^n+1)$, and $d_2=\gcd(k-1,p^n+1)$. If $d_1+d_2 \geq 5$, then the Weil spectrum over $L$ is at least five-valued.
Moreover, four of those five values are $\{0, -p^n, 2\alpha p^n, (2\beta+1)p^n\}$ where $\alpha, \beta \geq 1$. 
\end{theorem}

\begin{remark}\label{remarkk}
If $k=0$ or $1$ then $s$ is degenerate. So in general, we can take $2 \leq k \leq p^n$, since $k+p^n+1$ gives the same exponent $s$ $\pmod{p^{2n}-1}$ as $k$ over $L$.

For the case of $n=1$ in \cref{theoremM_i}, taking integer $k$ such that $p^{1/2}>2(k-1)$ would yield the same conclusion.
\end{remark}

\begin{theorem}\label{5valuesconjecturep^n3mod4}
Let $L$ be a quadratic extension of a finite field $F$ of order $p^n$, where $p$ is an odd prime. Let $s=1+k(p^n-1)$ be an invertible Niho exponent over $L$, $d_1=\gcd(k, p^n+1)$, and $d_2=\gcd(k-1,p^n+1)$. If $d_1+d_2=3$ and $p^n \equiv 11 \pmod{12}$, then the Weil spectrum over $L$ is at least five-valued. Moreover, four of those five values are $\{0, -p^n, p^n, 2 p^n\}$.
\end{theorem}
\bk

\section{Preliminaries}
We review some techniques with characters over a general finite field $F$ of order $q=p^n$.

For $a\in F$, let $\mu_a(x)=\mu(ax)$. Then the set of additive characters $\{ \mu_a : a \in F\}$ form an orthonormal  basis, with respect to the following inner product, for the space of functions from $F$ to $\mathbb{C}^\times$. One observes that the additive character $\mu(x)$ in our introduction is $\mu_1(x)$. 

\begin{definition}
For all functions $f,g: F \to \mathbb{C}^\times$, we define the inner product
\[
\langle f, g \rangle = \frac{1}{q}\sum_{x \in F}f(x) \overline{g(x)},
\] where $\overline{\,  \cdot \, }$ stands for complex conjugation.
\end{definition}

If we let $f_s$ be the function 
$f_s(x):=\mu(x^s)$, our Weil sum is the coordinates (or the Fourier coefficients) up to a factor of $1/q$ of $f_s$ with respect to the orthonormal basis  $\{ \mu_a : a \in F\}$. More precisely, the Weil sum becomes 
\[
W_{F,s}(a)=\sum_{x \in F}\mu(x^s-ax)=\sum_{x \in F}\mu(x^s)\overline{\mu_a(x)}=q\cdot \langle f_s, \mu_a \rangle,
\]
and
$$f_s=\frac{1}{q}\sum_{a \in F} W_{F,s}(a) \cdot \mu_a.$$
On the other hand,  
$$
 \langle f_s, f_s \rangle= 1,
$$
and hence,
\begin{align*}
    1=\left\langle \sum_{a \in F} \frac{1}{q}W_{F,s}(a)\cdot \mu_a,\sum_{b \in F}\frac{1}{q}W_{F,s}(b)\cdot \mu_b \right\rangle =\frac{1}{q^2}\sum_{a \in F}\bl |W_{F,s}(a)|^2. \bk
\end{align*}

\bl The Weil sum is shown to only take real values \cite[Theorem $2.1$(c)]{katz}, so the relation above becomes \bk
\begin{align}\label{secondpowermoment} 
    1=\frac{1}{q^2}\sum_{a \in F} W_{F,s}(a)^2.
\end{align}
Relation \eqref{secondpowermoment} can also be proved using the cross-correlation function in  \cite{katz}. In fact, it is called the second power moment of the Weil sum. In general we can consider the summation of all Weil \bl sums \bk in the finite field raised to a positive integer $m$. This is called the $m$th power moment. For the first few moments we have the following result which was proved in \cite{katz}.

\begin{lemma}\label{powermoment}
Let $F$ be a finite field of order $p^n$ and $s$ be a fixed invertible exponent. Then
\begin{itemize}
    \item[(i)] $\displaystyle\sum_{a \in F} W_{F,s}(a)=p^n$,
    \item[(ii)] $\displaystyle\sum_{a \in F} W_{F,s}(a)^2=p^{2n}$, and
    \item[(iii)] $\displaystyle\sum_{a \in F} W_{F,s}(a)^3=p^{2n} \cdot |R|$, where $R=\{x \in \bll F \bk \mid  (1-x)^s+x^s-1=0\}$.
    
\end{itemize}

\end{lemma}

As for the settings of a quadratic extension $L$ over $F$, we have the following moment property of the Weil sum in different orbits under the multiplication action of $F^{\times}$ on $L^{\times}$. 

\begin{lemma}\label{firstmoment}
Let $F$ be a finite field of order $p^n$ and $L$ be a quadratic extension of $F$. Suppose that $s$ is an invertible exponent \bl over $L$ \bk and $s\equiv 1\pmod{\bl p^n-1 \bk}$.
Then for a fixed $b \in L^{\times}$,
\[
\sum_{a\in F} W_{L,s}(ab)=
\begin{cases}
\rd p^{2n} & \text{if $b \in F$,}\\
0 & \text{otherwise.}
\end{cases}
\] 
\end{lemma} 

\begin{proof}
The first case for $b \in F$ was proved in \cite[\bll Lemma \bk 2.5]{cyclotomy}. So we will show the second equality here.
Observe that
\begin{align*}
    \sum_{a \in F}W_{L,s}(ab) &= \sum_{a \in F}\sum_{x \in L} \mu(x^s-abx)\\ 
    &=\sum_{x \in L} \mu(x^s)\sum_{a \in F}\zeta_p^{\Tr_{F/\mathbb{F}_p}(-a(Tr_{L/F}(bx)))}. 
\end{align*}

If $\Tr_{L/F}(bx) \neq 0$, then the inner sum is \bl the sum of all $p$th roots of unity precisely $p^{n-1}$ times, so it is $0$. \bk

Hence,
\begin{align*}
    \sum_{a\in F}W_{L,s}(ab) &=p^n \cdot \sum_{\substack{x \in L \\ \Tr_{L/F}(bx)=0}}\mu(x^s).
\end{align*}

\bl
Now we consider the equation   $0=\Tr_{L/F}(y)=y^{p^n}+y=y(y^{p^n-1}+1)$ over $L$. Note that the polynomial $y^{p^n}+y$ has formal derivative of $1$ so it is separable over $L$ with $p^n$ distinct roots. 
\bk
Let $x_0$ be a non-zero element such that $bx_0$ is a nonzero solution to $\Tr_{L/F}(y)=0$. Then all the roots of the polynomial are of the form $cbx_0$, where $c \in F$. Note that $\Tr_{L/F}(x_0)\neq 0$ because $b\not\in F$. 

\rd Now, suppose that $0=\Tr_{L/F}(x_0^s)=x_0^s(1+x_0^{s(p^n-1)})$. This means $x_0^{s(p^n-1)}=-1$ since $x_0$ is nonzero. Then $x_0^{p^n-1}=(-1)^{1/s}$, where $1/s$ is the inverse of $s$ modulo $p^{2n}-1$. If $p$ is odd, then $1/s$ is odd and $x_0^{p^n-1}=-1$, which contradicts $\Tr_{L/F}(x_0)\neq 0$. If $p=2$, then $x_0^{p^n-1}=1$. This also contradicts $\Tr_{L/F}(x_0)\neq 0$ in $L=\mathbb{F}_{2^{2n}}$. Therefore, $\Tr_{L/F}(x_0^s) \neq 0$.
\bk 

Hence, 
\begin{align*}
       \sum_{a\in F}W_{L,s}(ab) &= p^n \cdot \sum_{c \in F} \mu((cx_0)^s) \\
       &=p^n \cdot \sum_{c \in F} \zeta_p^{\Tr_{F/\mathbb{F}_p}(c\Tr_{L/F}(x_0^s))}=0.
\end{align*}

\bl

Note that the second-to-last equality follows from $s \equiv 1 \pmod{p^n-1}$ and $c^s=c$ in $F$.

\bk
\end{proof}

The conclusion of \cref{firstmoment} also implies the first moment property of the Weil sum.

\section{The Vanishing Conjecture and Bounds on $W_{L,s}(a)$}
For the rest of this paper (sections 3 and 4) we turn our focus to Niho exponents. \bl Let $p$ be an odd prime \bk. Our setting is a finite field $F$ of order $q=p^n$, together with a quadratic extension $L$ over $F$. Let $s$ be a nondegenerate invertible exponent. By \cref{powerlemma}, we can take $s \equiv 1 \pmod{p^n-1}$, then $s=k(p^n-1)+1$ for some nonnegative integer $k$. 

In this section, we give an alternative proof of the Vanishing Conjecture for the case of Niho exponent $s$ (i.e., \cref{vanishingp^2n}) to \cite{helleseth_lahtonen_rosendahl}, and study the bounds on the Weil sum for various values of $a$ and $s$. We first start with a lemma that gives a formula for Weil sum $W_{L,s}(a)$ based on the cardinality of a relevant set.

\begin{lemma}\label{divisibilityp^nlemma}
Let $L$ be the quadratic extension of the finite field $F$. Assume that $s$ is an invertible Niho exponent over $L$. Let $K_{a,s}=\{x \in L^\times\mid \Tr_{L/F}(x^s- a x)=0\}$. 

Then $|K_{a,s}|$ is a multiple of $(p^n-1)$ and 
$$
W_{L,s}(a)=\displaystyle p^n\cdot \frac{|K_{a,s}|}{p^n-1}-p^n.
$$
Furthermore, $W_{L,s}(a)$ is divisible by $p^n$.
\end{lemma}

\begin{remark}
The first statement of the theorem was also proved in \cite{cyclotomy}. 
\end{remark}

\begin{proof}

By \cref{powerlemma}, we can replace the condition $s \equiv p^j \pmod{p^n-1}$ by $s \equiv 1 \pmod{p^n-1}$.

\rd As seen in the proof of \cref{firstmoment}, the equation $y^{p^n}+y=0$ has $p^n$ distinct roots over $L$. \bk Hence, $$
|K_{0,s}|=|\{x \in L^\times\mid (x^s)^{p^n} \rd + \bk x^s=0\}| =p^n-1.
$$
So the identity $0=W_{L,s}(0)=\displaystyle p^n\cdot \frac{|K_{0,s}|}{p^n-1}-p^n$ holds.  We now assume $a\neq 0$.
We have 
\begin{align*}
    W_{L,s}(a)=\sum_{x \in L^\times}\mu(x^s-ax)+\mu(0)=\sum_{x \in L^\times}\mu(x^s-ax)+1.
\end{align*}

For any $y \in L^\times$, we can write $y=bx$ for some $b\in F^{\times}$, and 
$$
  \Tr_{L/F}((bx)^s-a(bx))=\Tr_{L/F}(bx^s-abx)=b\Tr_{L/F}(x^s-ax). 
$$ Therefore, each element $y$ in the coset $\bar x:=xF^\times$ either lies in $K_{a,s}$ or not depending on whether $x$ lies in $K_{a,s}$ or not.  This implies that $|K_{a,s}|$ is a multiple of $|F^\times|=p^n-1$. 

We then rewrite $\displaystyle\sum_{ x \in L^\times}\mu(x^s-ax) $ as follows. 

\begin{align*}
    \sum_{ x \in L^\times}\mu(x^s-ax) &=\sum_{ \bl \bar x \in L^{\times}/F^{\times} \bk }\sum_{b \in F^\times}\zeta_p^{\Tr_{F/\mathbb{F}_p}(\Tr_{L/F}((bx)^s-a(bx)))}\\
    &=\sum_{\bl \bar x \in L^{\times}/F^{\times} \bk}\sum_{b \in F^\times}\zeta_p^{\Tr_{F/\mathbb{F}_p}(b(\Tr_{L/F}(x^s-ax)))}\\
    &=\sum_{\bl \bar x \in L^{\times}/F^{\times} \bk}\sum_{b \in F} \zeta_p^{\Tr_{F/\mathbb{F}_p}(b(\Tr_{L/F}(x^s-ax)))}-(p^n+1).\\
\end{align*}

If $x \notin K_{a,s}$, then for a fixed equivalence class {$\bar x$} 
the inner sum   $\displaystyle\sum_{b \in F}\zeta_p^{\Tr_{F/\mathbb{F}_p}(b(\Tr_{L/F}(x^s-ax)))}=\displaystyle\sum_{u \in F}\zeta_p^{\Tr_{F/\mathbb{F}_p}(u)}$ is $0$; otherwise it is $p^n$.

Thus,

$$
W_{L,s}(a)=\frac{p^n|K_{a,s}|}{p^n-1}-(p^n+1)+1= p^n\cdot \frac{|K_{a,s}|}{p^n-1}-p^n.
$$

This completes the proof. 

\end{proof}

Now we are ready to give a proof of \cref{vanishingp^2n}.

\begin{proof}[Proof of Theorem \ref{vanishingp^2n}]

By \cref{divisibilityp^nlemma}, $W_{L,s}(a)=p^n \cdot h_a$ for some $h_a \in \mathbb{Z}$. Specifically, $h_0=0$ since $W_{L,s}(0)=0$. Applying this and relation \eqref{secondpowermoment} to the setting of a field $L$ of order $q=p^{2n}$, we have
\begin{align}\label{secondpowermomentL}
    q=p^{2n}=\displaystyle\sum_{a \in L^{\times}}h^2_a.
\end{align}

If $h_a=0$ for some $a \in L^\times$, then the Vanishing conjecture holds. To prove this, we use proof by contradiction and assume that $h_a\neq 0$ for all  $a \in L^\times$.
If $|h_a|=1$ for all $a \in L^\times$, then from $\eqref{secondpowermomentL}$, we have that $q-1=q$, which is not possible.
So $|h_{a'}| \geq 2$ for some $a' \in L$,
then
\[
\sum_{a \in L^{\times}}h_a^2
\geq \sum_{\substack{a \in L^\times \\ a \neq a'}}h_a^2+2^2=(q-2)+4=q+2>q,
\]
which also contradicts $\eqref{secondpowermomentL}$.

So at least $W_{L,s}(a)=0$ for some $a \in L^\times$.
\end{proof}

As a consequence to \cref{vanishingp^2n}, the Vanishing Conjecture holds true for $\mathbb{F}_{p^2}$. 

\begin{corollary}
Suppose $s$ is an invertible exponent \bl over $\mathbb{F}_{p^2}$ \bk and $s \equiv 1 \pmod{(p-1)}$, then the Vanishing Conjecture holds for the field $\mathbb{F}_{p^2}$.
\end{corollary}

Lemma \ref{divisibilityp^nlemma} gives a formula for $W_{L,s}(a)$ based on the cardinality of the set $K_{a,s}$. By identifying field elements in $K_{a,s}$, we can bound $|K_{a,s}|$ in order to deduce bounds on $W_{L,s}(a)$.

\begin{proposition} \label{prop1}
Let $a \in F$ and \bl $p$ be an odd prime. \bk Suppose $x^{2(p^n-1)}=1$ and $x \notin F$, then $\Tr_{L/F}(x^s-ax)=0$.
\end{proposition}

\begin{proof}

Since $x^{2(p^n-1)}=1$ and $x \notin F$, $x^{p^n-1}=-1$. 

We have that $x^{(p^n-1)^2}=x^{p^{2n}-1-2(p^n-1)}=(x^{2(p^n-1)})^{-1}=1$.

Now,
\begin{align*}
    \Tr_{L/F}(x^s-ax) &= \Tr_{L/F}(x^s)-a\Tr_{L/F}(x)\\
                &= x^s+x^{sp^n}-a(x+x^{p^n})\\
                &=x^s(1+x^{(k(p^n-1)+1)(p^n-1)})-ax(1+x^{p^n-1})\\
                &=x^s(1+x^{p^n-1})-ax(1+x^{p^n-1})\\
                &=0.
\end{align*}
\end{proof}
Note that there are $2(p^n-1)$ solutions for the equation $x^{2(p^n-1)}=1$ in $L$, since $\gcd(2(p^n-1),p^{2n}-1)=2(p^n-1)$. This gives a bound on the size of $K_{a,s}$, hence a bound on the Weil sum. 

\begin{proposition} \label{prop2}
Let $p^n \equiv 2 \pmod{3}$. If $x^{3(p^n-1)}=1$, then $\Tr_{L/F}(x^s-x)=0$.
\end{proposition}

\begin{proof}

We have that
\begin{align}
\Tr_{L/F}(x^s-x) &= \Tr_{L/F}(x^s)-\Tr_{L/F}(x) \nonumber\\
                &= x^s+x^{sp^n}-(x+x^{p^n}) \nonumber\\
                &=x^{k(p^n-1)+1}(1+x^{(k(p^n-1)+1)(p^n-1)})-x(1+x^{p^n-1}) \nonumber\\
                &=x(x^{k(p^n-1)}+x^{(p^n-1)(2k+1)}-1-x^{(p^n-1)}), \label{trace1}\\ \nonumber
\end{align}
using the relation $x^{(p^n-1)^2}=x^{p^{2n}-1-2(p^n-1)}=x^{-2(p^n-1)}=x^{(p^n-1)}$.

If $k \equiv 0  \pmod{3}$ or $k \equiv 1 \pmod{3}$, then the expression \eqref{trace1} becomes $0$.

If $k \equiv 2 \pmod{3}$, then $s \equiv 0 \pmod{3}$, but $q-1=p^{2n}-1 \equiv 0 \pmod{3}$. So $\gcd(s,q-1) \geq 3$, which is a contradiction.
\end{proof}

\begin{theorem}\label{generalboundtheorem}
For an \bl odd prime $p$\bk, we have the following bounds on $W_{L,s}(a)$:
\begin{itemize}
    \item[(1)] If $a \in L$, then $W_{L,s}(a)\geq-p^n$.
    \item[(2)] If $a \in F$, then $W_{L,s}(a) \geq 0$.
    \item[(3)] In particular, $W_{L,s}(1) \geq p^n$.
                If $p^n \equiv 2 \pmod{3}$, then $W_{L,s}(1) \geq 3p^n$. 
\end{itemize}
\end{theorem}

\begin{remark}
Since $W_{L,s}(a)$ is a sum of roots of unity, $|W_{L,s}(a)| < q$ for nondegenerate $s$. 
\end{remark}

\begin{proof}
Since $|K_{a,s}|\geq 0$, $|W_{L,s}(a)| \geq-q$ for $a \in L$.

If $a \in F$, then by \cref{prop1} there are at least $2(p^n-1)-(p^n-1)=p^n-1$ points in $K_{a,s}$. So $W_{L,s}(a) \geq 0$ by \cref{divisibilityp^nlemma}.

For part $(3)$, if $x \in F$, then $x^s=x$ and $\Tr_{L/F}(x^s-x)=0$. So such $x$ lies in $K_{1,s}$. Combining this fact and \cref{prop1}, there are at least $2(p^n-1)$ points in $K_{1,s}$. Therefore, $W_{L,s}(1) \geq p^n$. Moreover, if $p^n \equiv 2 \pmod{3}$, then there are $3(p^n-1)$ solutions to the equation $x^{3(p^n-1)}=1$, and by \cref{prop2} and \cref{divisibilityp^nlemma}, $W_{L,s}(1) \geq 3p^n$. 
\end{proof}

\section{The Weil Spectrum}
\bl Let $L$ be quadratic extension of a finite field $F$ of order $p^n$, where $p$ is an odd prime. \bk In this section we give formulas for the Weil sum at specific elements \bl of $L$. Let $s$ be an invertible Niho exponent over $L$. \bk We recall that $s=1+k(p^n-1)$ and
as noted in \cref{remarkk}, we can take $2 \leq k \leq p^n$.
Our discussion begins by considering the value of the Weil sum at a root of unity in the field for certain primes $p$. The formula is obtained by realizing the relation between the elements in the set $K_{a,s}=\{x \in L^\times\mid\Tr_{L/F}(x^s-ax)=0\}$ in  \cref{divisibilityp^nlemma} and the root of unity.

\begin{proposition}\label{W1}
\bl Let $s=1+k(p^n-1)$ be an invertible Niho exponent over $L$, where $2 \leq k \leq p^n$. \bk Let $d_1=\gcd(k,p^n+1)$, $d_2=\gcd(k-1,p^n+1)$, and $t$ be a positive integer with  \bl $t\mid p^{n}+1$. \bk Let $\zeta_t$ be a primitive $t$-th root of unity in $L$. 
For $i=1$ or $2$, let
\[
\delta_{i,t}=\begin{cases} 
      1 & \text{if $t \mid \frac{p^n+1}{d_i}$,} \\
      0 & \text{otherwise.}
   \end{cases}
   \]
Then \[
W_{L,s}(\zeta_t)=\begin{cases} 
      p^n(d_1+d_2-2) & \text{if $t=1$,} \\
      p^n(d_1\delta_{1,t}+d_2\delta_{2,t}-1) & \text{otherwise.}
   \end{cases}
\]
\end{proposition}

\begin{proof}
We compute $|K_{\zeta_t,s}|$ in \cref{divisibilityp^nlemma}. Let $x \in K_{\zeta_t,s}$ then $\Tr_{L/F}(x^s)=\Tr_{L/F}(\zeta_tx)$.
We also have that $N_{L/F}(x^s)=N_{L/F}(\zeta_tx)$, since $\zeta^{p^n+1}_t=1$. Hence, $\zeta_tx$ and $x^s$ satisfy the same degree two minimal polynomial over $F$. So we can consider two cases $x^s=\zeta_tx$ or $x^s=(\zeta_tx)^{p^n}$.
Let $$L^{\times}=\langle g \rangle$$ for some generator $g$ in the field. Then $x=g^i$ for some $i \in \mathbb{Z}_{p^{2n}-1}$. We can pick $\zeta_t=g^{(p^{2n}-1)j/t}$ where $\gcd(j,t)=\gcd(j,p^{2n}-1)=1$.
For the case $x^s=\zeta_tx$, we have that $x^{k(p^n-1)}=\zeta_t$. Then $g^{ik(p^n-1)}=g^{(p^{2n}-1)j/t}$, so \begin{align}\label{w1_1}
    (p^n-1)ik \equiv \frac{(p^{2n}-1)j}{t} \pmod{p^{2n}-1},
\end{align} 
which implies that 
\begin{align}\label{w1_2}
    ik \equiv \frac{(p^n+1)j}{t} \pmod{p^n+1}.
\end{align}
Let $d_1=\gcd(k, p^n+1)$. Then \eqref{w1_2} is solvable if $\frac{p^n+1}{t} \equiv 0 \pmod{d_1}$. If it is solvable then there are $d_1$ solutions. When $t=1$, \eqref{w1_2} is always solvable. Hence \eqref{w1_1} has $d_1(p^n-1)$ solutions if $t=1$, and $d_1\delta_{1,t}(p^n-1)$ solutions otherwise.
Similarly, for the case $x^s=(\zeta_tx)^{p^n}=\zeta^{-1}_tx^{p^n}$, we have that $x^{(k-1)(p^n-1)}=\zeta^{-1}_t=g^{-(p^{2n}-1)j/t}$. Let $d_2=\gcd(k-1,p^n+1)$, then there are $d_2(p^n-1)$ solutions to this case if $t=1$ and $d_2\delta_{2,t}(p^n-1)$ for other values of $t$. 

\bl
For both cases to have simultaneous solutions, we have that $\zeta_t x=x^s=(\zeta_tx)^{p^n}$. This means $x^s=\zeta_t x \in F^{\times}$ and $x^{s(p^n-1)}=1$. \rd Since the power map $x \mapsto x^s$ permutes both $L$ and the subfield $F$, $x^s \in F$ if and only if $x \in F$. Therefore, we have $x \in F^{\times}$. \bk We also note that $\zeta_t=x^{s-1}=x^{k(p^n-1)}=1$. 

Hence, when $t=1$ the solutions for both cases were counted twice for $x \in F^{\times}$. Therefore,

\bk

\[
|K_{\zeta_t,s}|=\begin{cases}
    (p^n-1)(d_1+d_2-1) & \text{if $t=1$,}\\
    (p^n-1)(d_1\delta_{1,t}+d_2\delta_{2,t}) & \text{otherwise.}
\end{cases}
\]
Apply this to the formula for $W_{L,s}(\zeta_t)$ in \cref{divisibilityp^nlemma}, we have 
\[
W_{L,s}(\zeta_t)=\begin{cases} 
      p^n(d_1+d_2-2) & \text{if $t=1$,} \\
      p^n(d_1\delta_{1,t}+d_2\delta_{2,t}-1) & \text{otherwise.}
   \end{cases}
\]
\end{proof}

\begin{remark}
Theorem \ref{generalboundtheorem}(3), can be obtained by \cref{W1}. As noted in \cref{remarkp^n2mod3}, \bl for an odd prime $p$\bk, either $d_1$ or $d_2$ must be divisible by $2$, so $d_1+d_2 \geq 3$.
Moreover, if $p^n \equiv 2 \pmod{3}$, then $d_1+d_2 \geq 5$, and thus, $W_{L,s}(1)=p^n(d_1+d_2-2) \geq 3p^n$.
\end{remark}

From \cref{W1}, we deduce the following corollary for the Weil sum at $a=-1$.

\begin{corollary}\label{W-1}
    Let \bl $p$ be an odd prime, \bk $L$ be a quadratic extension \bl of order $p^{2n}$ \bk over \bl the finite field \bk $F$, $s$ be an Niho exponent \bl over $L$, \bk and $d_1,d_2$ be defined as in \cref{W1}. Then
\[
W_{L,s}(-1)=p^n\left(d_1 \cdot \frac{1+(-1)^{(p^n+1)/d_1}}{2}+d_2 \cdot \frac{1+(-1)^{(p^n+1)/d_2}}{2}-1\right).
\]
If $p^n \equiv 3 \pmod{4}$ \bl and \bk $d_1+d_2=3$, then $W_{L,s}(-1)=2p^n$.
\end{corollary}

We state a result from Katz in \cite{cyclotomy} that is useful in the next step of our discussion. 

\begin{lemma}\cite[Corollary 3.4]{cyclotomy}\label{cyclotomykatz}
Let $F$ be a finite field of characteristic $p$, and let $L$ be an extension of $F$ with $[L:F]$ a power of a prime $\ell$ distinct from $p$. Let $s$ be degenerate over $F$ but not over $L$. Then $W_{L,s}(1) \equiv |F| \pmod{\ell}$ and $W_{L,s}(a) \equiv 0 \pmod{\ell}$ for every $a \in F\setminus \{1
\}$.
\end{lemma}

Combining this lemma with our previous results, we can deduce the following four values in the Weil spectrum in the cases of \cref{5valuesconjecture}.

\begin{corollary}\label{4valued}
\bl Let $L$ be a quadratic extension of a finite field $F$ of order $p^n$, where $p$ is an odd prime. Let $s=1+k(p^n-1)$ be an invertible Niho exponent over $L$, $d_1=\gcd(k, p^n+1)$, and $d_2=\gcd(k-1,p^n+1)$. 
\begin{itemize}
    \item[(i)] If $d_1+d_2 \geq 5$, then the spectrum of $W_{L,s}(a)$ contains at least $4$ values of the form $\{0, -p^n, 2\alpha p^n, (2\beta+1)p^n\}$ for some integers $\alpha,\beta \geq 1$.
    \item[(ii)] If $d_1+d_2=3$ and $p^n \equiv 11 \pmod{12}$, then the spectrum of $W_{L,s}(a)$ contains $0, -p^n, 2p^n, p^n$.
\end{itemize} 
\bk
\end{corollary}

\begin{proof}
\bl
In both cases: By \cref{2values}, the Weil spectrum contains at least three values, and one of which is $0$ by \cref{vanishingp^2n}. \rd If all the nonzero values were positive, \bk we would have $\left(\displaystyle\sum_{a \in L}W_{L,s}(a)\right)^2 > \displaystyle\sum_{a \in L}W^2_{L,s}(a)$, and this is contradiction to the first and second moments in \cref{powermoment}. Hence, the spectrum must contain at least a negative value. From \cref{divisibilityp^nlemma} and \cref{generalboundtheorem} (part $1$), this negative value must be $-p^n$.

For case (i): Apply \cref{cyclotomykatz} to our setting of the quadratic extension $L$ over $F$, \bl the prime $\ell=[L:F]=2$. Then the Weil sum $W_{L,s}(a)$ admits an odd value for $a=1$ and even values for all $a \in F \setminus \{1\}$. By \cref{W1}, 
 $W_{L,s}(1)=(d_1+d_2-2)p^n \geq 3p^n$. If $W_{L,s}(a)=0$ for all $a \in F \setminus \{1\}$, then taking $b=1$ in \cref{firstmoment} yields $p^{2n}=\displaystyle\sum_{a \in F}W_{L,s}(a)=W_{L,s}(1)$. \rd Together with the value $-p^n$, this would mean $\displaystyle\sum_{a \in L}W_{L,s}^2(a) > p^{4n}$, contradicting to the second power moment or relation \bll \eqref{secondpowermoment}\bk. Hence, there is a nonzero even value for some $a \in F \setminus \{1\}$. Therefore, the Weil spectrum consists of $-p^n,0,2\a p^n$, and $(2\beta+1)p^n$, where $\a,\b \geq 1$.
 
 For case (ii): By \cref{W1}, 
 $W_{L,s}(1)=(d_1+d_2-2)p^n = p^n$. By \cref{W-1}, $W_{L,s}(-1)=2p^n$. Hence the Weil spectrum in this case consists of $-p^n, 0, p^n$, and $2p^n$.
\end{proof}

\bk

Our numerical evidence suggests a stronger conclusion than \cref{4valued} implies. This leads to \cref{5valuesconjecture} and \cref{5valuesconjecturep^n2mod3}.

For the rest of the paper we discuss partial results towards \cref{5valuesconjecture} for certain families of Niho exponents $s$. For such an $s$, we give a count on the solution set $R=\{x \in L\mid (1-x)^s+x^s-1=0\}$ in \cref{powermoment}.

\begin{lemma}\label{xx^slemma}
Let $L$ be a quadratic extension of $F$, and $|F|=p^n$ and $k \geq 2$. Let $d_1=\gcd(k,p^n+1)$ and $d_2=\gcd(k-1,p^n+1)$. Then $|R|=p^n+(d_1-1)(d_1-2)+(d_2-1)(d_2-2)$.
\end{lemma} 

\begin{proof}
Clearly, all elements in $F$ are in $R$. So $|R| \geq p^n$. Now suppose $x \in R \setminus F$.

We have that $(1-x)^s=1-x^s$. Computing the norm $N_{L/F}$ of both sides, we get
\[
N_{L/F}(1-x^s)=1-x^{sp^n}-x^s+x^{s(p^n+1)}=1-\Tr_{L/F}(x^s)+N_{L/F}(x^s)
\]
and
\[
N_{L/F}((1-x)^s)=1-x^{p^n}-x+x^{p^n+1}=1-\Tr_{L/F}(x)+N_{L/F}(x).
\]

As $s=1+k(p^n-1)$, we know $N_{L/F}(x)=N_{L/F}(x^s)$. Equating the norm of $1-x^s$  gives us $\Tr_{L/F}(x)=\Tr_{L/F}(x^s)$.

Since the norm and trace of $x$ and $x^s$ are the same, they must satisfy the same degree-two minimal polynomial over $F$. Hence $x^s=x$ or $x^s=x^{p^n}$. 

\textbf{Case 1}: $x^s=x$.

This implies $x^{k(p^n-1)}=1$. Since $x \notin F$, $x^{p^n-1} \neq 1$. Now, $1=x^{p^{2n}-1}=x^{(p^n-1)(p^n+1)}$. So a solution in this case must satisfy $x^{d_1({p^n-1})}=1$, where $d_1=\gcd(k,p^n+1)$. Let $L^{\times}= \langle g \rangle$ and $h=g^{(p^{2n}-1)/d_1}$ be an element of order $d_1$ in $L^\times$. Then $x^{p^n-1}$ must be in $\langle h \rangle$. Without loss of generality, let $x^{p^n-1}=h^{t_1}$, where $ 1\leq t_1 \leq d_1-1$. 
On the other hand, $1=x^s+(1-x)^s=x+(1-x)^s$. This implies $(1-x)^{s-1}=1$ or $(1-x)^{k(p^n-1)}=1$. Using the similar argument from above, we can say $(1-x)^{p^n-1}=h^{t_2}$, where $1 \leq t_2 \leq d_1-1$. \bll Since $x \notin F$, $t_1 \neq t_2$. \bk

Now,
\begin{align}\label{p^n-1}
     (1-x)^{p^n-1} = h^{t_2} 
    &\implies 1-x^{p^n} = h^{t_2}(1-x)\\ \nonumber
    &\implies 1-h^{t_1}x =h^{t_2}-h^{t_2}x\\ \nonumber
    &\implies  x=\frac{1-h^{t_2}}{h^{t_1}-h^{t_2}}.\nonumber 
\end{align} 
With $1 \leq t_1,t_2 \leq d_1-1$ and $t_1 \neq t_2$, there are $(d_1-1)(d_1-2)$ choices for solution $x$. \bll To see all these choices are distinct, suppose there are integer pairs $(t_1,t_2), (u_1,u_2)$, where $1 \leq t_1,t_2,u_1,u_2 \leq d_1-1$, $t_1 \neq t_2$, and $u_1 \neq u_2$, satisfying
\begin{align*}
x=\frac{1-h^{t_2}}{h^{t_1}-h^{t_2}}=\frac{1-h^{u_2}}{h^{u_1}-h^{u_2}}.
\end{align*}
Since $h$ has order $d_1=\gcd(k,p^n+1)$, $h^k=h^{p^n+1}=1$. Thus, $h^{p^n}=h^{-1}$. Using this fact we compute
\begin{align*}
x^{p^n}&=\left(\frac{1-h^{t_2}}{h^{t_1}-h^{t_2}}\right)^{p^n} =\frac{1-h^{t_2p^n}}{h^{t_1p^n}-h^{t_2p^n}}=\frac{1-h^{-t_2}}{h^{-t_1}-h^{-t_2}}=h^{t_1}x.
\end{align*}
Hence, $x^{p^n-1}=h^{t_1}$ and similarly, we find $x^{p^n-1}=h^{u_1}.$ This means $t_1=u_1$ since $h$ has order $d_1$. From here, we can reverse the implications of \eqref{p^n-1} to show $h^{t_2}=(1-x)^{p^n-1}=h^{u_2}$, which means $t_2=u_2$.

\bk

Now, to show such $(d_1-1)(d_1-2)$ choices are in $R$, consider $x=\displaystyle\frac{1-h^{t_2}}{h^{t_1}-h^{t_2}}$, where $h$ is defined as above, and $1 \leq t_1, t_2 \leq d_1-1$.

We have that
\begin{align}\label{denom}
  (h^{t_1}-h^{t_2})^s=(h^{t_1}-h^{t_2})^{(1+k(p^n-1))} &= (h^{t_1}-h^{t_2})^{(1-k)}(h^{t_1p^n}-h^{t_2p^n})^k \\ \nonumber
  &= (h^{t_1}-h^{t_2})^{(1-k)}(h^{-t_1}-h^{-t_2})^k \\ \nonumber
  &=h^{-t_1k}h^{-t_2k}(h^{t_1}-h^{t_2})^{(1-k)}(h^{t_2}-h^{t_1})^k \\ \nonumber
  &=(h^{k})^{-t_1}(h^{k})^{-t_2}(-1)^k(h^{t_1}-h^{t_2}) \\ \nonumber
  &=(-1)^k(h^{t_1}-h^{t_2}) \\
\end{align}

Similarly, we can show that
\begin{align}\label{num1}
    (h^{t_1}-1)^s=(-1)^k(h^{t_1}-1),
\end{align}
and
\begin{align}\label{num2}
(1-h^{t_2})^s=(-1)^k(1-h^{t_2}).
\end{align}

Now, 
\begin{align*}
    (1-x)^s+x^s-1 &= \left(1-\frac{1-h^{t_2}}{h^{t_1}-h^{t_2}}\right)^s+\left(\frac{1-h^{t_2}}{h^{t_1}-h^{t_2}}\right)^s-1 \\
    &=\left(\frac{(h^{t_1}-1)^s+(1-h^{t_2})^s}{(h^{t_1}-h^{t_2})^s}\right)-1=0,
\end{align*}
with the last equality following from relations \eqref{denom}, \eqref{num1}, and \eqref{num2}.
Hence, such choice $x$ is in $R$.

\bk

\textbf{Case 2}: $x^{p^n}=x^s=x^{1+k(p^n-1)}$.

This implies $x^{(k-1)(p^n-1)}=1$. Suppose $x \notin F$. Similar to the argument in case $1$ we let $d_2=\gcd(k-1,p^n+1)$ and find solutions to the equation $x^{d_2(p^n-1)}=1$. \bl Let $\ell=g^{(p^{2n}-1)/d_2}$ be an element of order $d_2$, where $g$ is the generator of $L^{\times}$ as in case $1$. A quick check also yields us the relation $(1-x)^{(k-1)(p^n-1)}=1$.Then $x^{p^n-1}$ and $(1-x)^{p^n-1}$ must be in $\langle \ell \rangle$. Using the similar argument from above, we obtain $x=\displaystyle\frac{1-\ell^{r_2}}{\ell^{r_1}-\ell^{r_2}}$, where $1 \leq r_1,r_2 \leq d_2-1$.
There are $(d_2-1)(d_2-2)$ choices for such $x$. \bll Using similar arguments as in Case 1, we can show that all these choices are distinct.

\bk

In the reverse direction, we first note that $\ell^{k}=\ell$ and $\ell^{p^n}=\ell^{-1}$, since $\ell$ has order $d_2=\gcd(k-1,p^n+1)$.

Using a similar argument as in case $1$, we have that
\begin{align}\label{denomcase2}
  (\ell^{r_1}-\ell^{r_2})^s &= (\ell^{r_1}-\ell^{r_2})^{(1-k)}(\ell^{r_1p^n}-\ell^{r_2p^n})^k \\ \nonumber
  &= (\ell^{r_1}-\ell^{r_2})^{(1-k)}(\ell^{-r_1}-\ell^{-r_2})^k \\ \nonumber
  &=\ell^{-r_1k}\ell^{-r_2k}(\ell^{r_1}-\ell^{r_2})^{(1-k)}(\ell^{r_2}-\ell^{r_1})^k \\ \nonumber
  &=(-1)^k\ell^{-r_1}\ell^{-r_2}(\ell^{r_1}-\ell^{r_2}) \\ \nonumber
  &=(-1)^k(\ell^{-r_2}-\ell^{-r_1}).
\end{align}

Similarly, we can show that
\begin{align}\label{num1case2}
    (\ell^{r_1}-1)^s=(-1)^k(1-\ell^{-r_1}),
\end{align}
and
\begin{align}\label{num2case2}
(1-\ell^{r_2})^s=(-1)^k(\ell^{-r_2}-1).
\end{align}

By relations \eqref{denomcase2}, \eqref{num1case2}, and \eqref{num2case2}, we can show that $(1-x)^s+x^s-1=0$. Thus, such choice $x$ is in $R$.

\bk

Hence there are $(d_2-1)(d_2-2)$ solutions of $x$ in this case.

Note that since $k$ and $k-1$ are coprime, $d_1$ and $d_2$ are coprime as well. Therefore, the solutions in case $1$ and case $2$ for $x \notin F$ are distinct. Accounting for solutions $x \in F$ we have $|R|=p^n+(d_1-1)(d_1-2)+(d_2-1)(d_2-2)$.
\end{proof}

\begin{corollary}
Let $s=1+k(p^n-1)$ for some integer $k$, $0 \leq k \leq p^n$. Then $k$ and $2-k+p^n$ gives the same number of solutions to the equation $(1-x)^s+x^s-1=0$ for $x \in L$.
\end{corollary}

\begin{proof}
From the proof of \cref{xx^slemma}, the exponents  $s$ and $sp^n$ give the same number of solutions to the equation $(1-x)^s+x^s-1=0$ for $x \in L$. Now,  $sp^n \equiv 1+(1-k)(p^n \rd -1) \pmod{p^{2n}-1}$ \bk, and $0 \leq 2-k+p^n \leq p^n$ gives the same exponent \rd modulo $(p^{2n}-1)$ \bk as $1-k$ \rd over $L$. \bk
\end{proof}

\bl 

Finally, we end our discussion by giving some partial results towards \cref{5valuesconjecture}.
The idea behind the proofs of \cref{5valuesconjecturep^n3mod4} and \cref{5valuesconjecturep^n2mod3} is to apply the power moments in \cref{powermoment} to the four Weil sum values and the count for the set $R$ in \cref{xx^slemma} to derive a contradiction.

\bk

\begin{proof}[Proof of Theorem \ref{theoremM_i}]
\bl According to \cref{4valued}, there are four values in the Weil spectrum. Suppose that these are the only ones in the spectrum. \bk Let $m_1,m_2,m_3$ and $m_4$ be the number of elements whose Weil sum value is $-p^n, 0, 2\alpha p^n$ and $(2 \beta +1)p^n$, respectively, for integers $\alpha, \beta \geq 1$. In here $2 \beta+1=d_1+d_2-2$ from \cref{W1}. 
By \cref{powermoment} we have the following system of equations: 
\begin{numcases}
 \ m_1+m_2+m_3+m_4=p^{2n} \\ \label{1}
 -m_1+2\alpha m_3 + (d_1+d_2-2)m_4=p^n \\ \label{2}
 m_1+4{\alpha}^2 m_3+(d_1+d_2-2)^2 m_4 =p^{2n}\\ \label{3}
 -m_1+8{\alpha}^3 m_3 +(d_1+d_2-2)^3 m_4=p^n |R|, \label{4}
\end{numcases}
where $|R|=p^n+(d_1-1)(d_1-2)+(d_2-1)(d_2-2)$. 

The above system has a unique solution over $\mathbb{Q}$, which is
\begin{align*}
m_1 &= \bl -\displaystyle\frac{p^n(d_1^2 + d_2^2+(2\a-p^n-3)(d_1+d_2) - p^n(2\a-3) + 4(1-\a))}{(2\a + 1)(d_1 + d_2 - 1)} \bk \nonumber \\
m_2 &=\displaystyle\frac{1}{2} \cdot\displaystyle\frac{p^n((d_1+d_2)(2\a(p^n+1)-p^n-4)+d_1^2+d_2^2-6\a(p^n+1)+2(2p^n+3))}{\a(d_1 + d_2 - 2)} \nonumber \\
m_3 &=\displaystyle\frac{1}{2} \cdot\displaystyle\frac{p^n(d_1^2 + d_2^2 - (p^n+4)(d_1 + d_2) + 2(2p^n + 3))}{(2\a - d_1 - d_2 + 2)(2\a + 1)\a} \nonumber \\
m_4 &= -\displaystyle\frac{p^n(d_1^2 + d_2^2 - 2p^n(\a-1) - 3(d_1 + d_2) + 4-2\a)}{(2\a - d_1 - d_2 + 2)(d_1 + d_2 - 1)(d_1 + d_2 - 2)}.
\end{align*}
\bl 
From the numerator of $m_3$, we have that 
\begin{align}
d_1^2 + d_2^2 - (p^n+4)&(d_1 + d_2) + 2(2p^n + 3) \nonumber \\ &=(d_1-2)^2+(d_2-2)^2-(d_1+d_2-4)p^n \bl - 2 \nonumber \\
&\leq {(d_1-2)^2+(d_2-2)^2-p^n \bl - 2}, \label{m1} 
\end{align}
since $d_1+d_2-4\geq 1$. 
Note that since $k < \displaystyle\frac{p}{2}+1$, $d_1<\displaystyle\frac{p}{2}+1$ and $d_2 <\displaystyle\frac{p}{2}$. Hence, we can bound \eqref{m1} by  $\left(\displaystyle\frac{p}{2}-1\right)^2+\left(\displaystyle\frac{p}{2}-2\right)^2-p^n-2 = \displaystyle\frac{1}{2}p^2-3p-p^n+3 \leq -\displaystyle\frac{1}{2}p^2-3p+3<0$, since $n \geq 2$. So the numerator of $m_3$ is negative. Thus, the denominator of $m_3$, i.e $2(2\a - d_1 - d_2 + 2)(2\a + 1)\a$ is negative, which implies the factor $(2\a - d_1 - d_2 + 2)<0$.

Now, this forces the denominator of $m_4$ to be negative, which implies that the expression $(d_1^2 + d_2^2 - 2p^n(\a-1) - 3(d_1 + d_2) + 4-2\a)$ in the numerator of $m_4$ must be positive. If $\a \geq 2$, using the bounds for $d_1$ and $d_2$, we can bound the numerator of $m_4$ by
\begin{align*}
    (d_1^2 + d_2^2 - 2p^n(\a-1) &- 3(d_1 + d_2) + 4-2\a) \\ &\leq d_1^2+d_2^2-2p^n-3(d_1+d_2) \\
    &=\left(d_1-\frac{3}{2}\right)^2+\left(d_2-\frac{3}{2}\right)^2-\frac{9}{2}-2p^n\\
    &< \left(\frac{p+2}{2}-\frac{3}{2}\right)^2+\left(\frac{p}{2}-\frac{3}{2}\right)^2-\frac{9}{2}-2p^n\\
    &=\frac{1}{2}p^2-2p-2p^n-2 \leq -\frac{1}{2}p^2-2p-2<0,
\end{align*}
which is a contradiction.

Hence, $\a$ must be 1. 
Replacing this for $m_4$, we have $$m_4=\displaystyle\frac{1}{2}\displaystyle\frac{p^n(d_1^2+d_2^2-3(d_1+d_2)+2)}{(d_1+d_2-4)(d_1+d_2-2)(d_1+d_2-1)}.$$
Observe that the factors in the denominator $(d_1+d_2-4)<(d_1+d_2-2)<(d_1+d_2-1)<\displaystyle\frac{p+2}{2}+\frac{p}{2}-1=p$. Moreover, $(d_1+d_2-2), (d_1+d_2-1) \geq 3$, so these two factors do not divide $p$. Hence, for $m_4$ to be an integer, they must divide $d_1^2+d_2^2-3(d_1+d_2)+2$. 

However, 
\begin{align*}
    d_1^2+d_2^2-3(d_1+d_2)+2&=(d_1+d_2-4)(d_1+d_2-1)-2(d_1-1)(d_2-1)\\
    &<(d_1+d_2-2)(d_1+d_2-1).\\
\end{align*}
This is a contradiction.

Therefore, there must be a fifth value in this Weil spectrum.
\end{proof}

\bl 

\begin{remark}
For the case $n=1$ in \cref{theoremM_i}, taking odd prime p such that $p^{1/2}>2(k-1)$, and following the argument in the proof of \cref{theoremM_i} with this bound would yield the same conclusion (i.e the Weil spectrum has at least five values).
\end{remark}

\bl

\begin{proof}[Proof of Theorem \ref{5valuesconjecturep^n3mod4}]
This proof is in a similar flavor to the proof of \cref{theoremM_i}. Since $p^n \equiv 11 \pmod{12}$, $p^n \equiv 3 \pmod{4}$. By \cref{W-1}, $W_{L,s}(-1)=2p^n$. Since $d_1+d_2=3$, $W_{L,s}(1)=p^n$ by \cref{W1}. As in the last proof, let $m_1,m_2,m_3$ and $m_4$ be the number of elements whose Weil sum value is $-p^n, 0, 2\alpha p^n$ and $(2 \beta +1)p^n$, but here we take $\a=1,\b=0$, specifically.
Now, $d_1+d_2=3$, where $d_1, d_2 \geq 1$ (since $k \geq 2$) implies that one of them is $1$ and the other one is $2$. Hence, $d_1^2+d_2^2=5$.
Replacing these values of $\a,\b, d_1+d_2=3$, and $d_1^2+d_2^2=5$ in the solutions of $m_1,m_2,m_3$, and $m_4$ in the proof of \cref{theoremM_i}, we obtain $
m_1 = \displaystyle\frac{p^n(p^n-1)}{3}, 
m_2 =\displaystyle\frac{p^n(p^n-1)}{2}, 
m_3 =\displaystyle\frac{p^n(p^n-1)}{6},$ and $m_4 = p^n.$

Since $p^n \equiv 11 \pmod{12}$, $p \neq 3$. Since $m_1$ must be an integer, $p^n \equiv 1 \pmod{3}$, but this is a contradiction. Therefore, there exists a fifth value in this Weil spectrum.

\end{proof}

\bk

\section*{Acknowledgements}
The author would like to thank  Qing Xiang for suggesting the problem, Daniel Katz, \bl Wen-Ching Winnie Li, \bk Fang-Ting Tu and Bao Pham for helpful comments on an earlier draft to this paper, and Ling Long for fruitful discussions and feedback to this project. \bl The author is also grateful for the elaborate comments from the anonymous referee. All suggestions have improved this paper greatly. \bk

\bibliographystyle{amsplain}
\bibliography{ReferenceFull}

\providecommand{\bysame}{\leavevmode\hbox to3em{\hrulefill}\thinspace}
\providecommand{\MR}{\relax\ifhmode\unskip\space\fi MR }
\providecommand{\MRhref}[2]{%
  \href{http://www.ams.org/mathscinet-getitem?mr=#1}{#2}
}
\providecommand{\href}[2]{#2}
\begin{thebibliography}{10}

\bibitem{cyclotomy}
Yves Aubry, Daniel~J. Katz, and Philippe Langevin, \emph{Cyclotomy of {W}eil
  sums of binomials}, J. Number Theory \textbf{154} (2015), 160--178.
  \MR{3339571}

\bibitem{Calderbank3value}
A.~R. Calderbank, Gary McGuire, Bjorn Poonen, and Michael Rubinstein, \emph{On
  a conjecture of {H}elleseth regarding pairs of binary {$m$}-sequences}, IEEE
  Trans. Inform. Theory \textbf{42} (1996), no.~3, 988--990. \MR{1445885}

\bibitem{WelchConjecture}
Anne Canteaut, Pascale Charpin, and Hans Dobbertin, \emph{Binary
  {$m$}-sequences with three-valued crosscorrelation: a proof of {W}elch's
  conjecture}, IEEE Trans. Inform. Theory \textbf{46} (2000), no.~1, 4--8.
  \MR{1743572}

\bibitem{charpin}
Pascale Charpin, \emph{Cyclic codes with few weights and {N}iho exponents}, J.
  Combin. Theory Ser. A \textbf{108} (2004), no.~2, 247--259. \MR{2098843}

\bibitem{feng}
Tao Feng, \emph{On cyclic codes of length {$2^{2^r}-1$} with two zeros whose
  dual codes have three weights}, Des. Codes Cryptogr. \textbf{62} (2012),
  no.~3, 253--258. \MR{2886276}

\bibitem{games}
Richard~A. Games, \emph{The geometry of {$m$}-sequences: three-valued
  crosscorrelations and quadrics in finite projective geometry}, SIAM J.
  Algebraic Discrete Methods \textbf{7} (1986), no.~1, 43--52. \MR{819704}

\bibitem{Tormaster}
Tor Helleseth, \emph{Krysskorrelasjonsfunksjonen mellom maksimale sekvenser
  over {$GF$}($q$)}, Master's thesis, Universitetet i Bergen, 1971.

\bibitem{tor}
Tor Helleseth, \emph{Some results about the cross-correlation function between
  two maximal linear sequences}, Discrete Math. \textbf{16} (1976), no.~3,
  209--232. \MR{0429323}

\bibitem{helleseth_lahtonen_rosendahl}
Tor Helleseth, Jyrki Lahtonen, and Petri Rosendahl, \emph{On {N}iho type
  cross-correlation functions of {$m$}-sequences}, Finite Fields Appl.
  \textbf{13} (2007), no.~2, 305--317. \MR{2307130}

\bibitem{WelchNihoConjecture}
Henk D.~L. Hollmann and Qing Xiang, \emph{A proof of the {W}elch and {N}iho
  conjectures on cross-correlations of binary {$m$}-sequences}, Finite Fields
  Appl. \textbf{7} (2001), no.~2, 253--286. \MR{1826337}

\bibitem{katz}
Daniel~J. Katz, \emph{Weil sums of binomials, three-level cross-correlation,
  and a conjecture of {H}elleseth}, J. Combin. Theory Ser. A \textbf{119}
  (2012), no.~8, 1644--1659. \MR{2946379}

\bibitem{divisibility}
\bysame, \emph{Divisibility of {W}eil sums of binomials}, Proc. Amer. Math.
  Soc. \textbf{143} (2015), no.~11, 4623--4632. \MR{3391022}

\bibitem{Katz2018WeilSO}
Daniel~J. Katz, \emph{Weil sums of binomials: properties, applications, and
  open problems}, Combinatorics and Finite Fields: Difference Sets,
  Polynomials, Pseudorandomness and Applications, 23, De Gruyter, 2019,
  pp.~109--134.

\bibitem{threevaluefamily}
Daniel~J. Katz and Philippe Langevin, \emph{Proof of a conjectured three-valued
  family of {W}eil sums of binomials}, Acta Arith. \textbf{169} (2015), no.~2,
  181--199. \MR{3359953}

\bibitem{katz_langevin}
\bysame, \emph{New open problems related to old conjectures by {H}elleseth},
  Cryptogr. Commun. \textbf{8} (2016), no.~2, 175--189. \MR{3488215}

\bibitem{Nihoexponent}
Nian Li and Xiangyong Zeng, \emph{A survey on the applications of {N}iho
  exponents}, Cryptogr. Commun. \textbf{11} (2019), no.~3, 509--548.
  \MR{3946534}

\bibitem{mcguire}
Gary McGuire, \emph{On certain 3-weight cyclic codes having symmetric weights
  and a conjecture of {H}elleseth}, Sequences and their applications ({B}ergen,
  2001), Discrete Math. Theor. Comput. Sci. (Lond.), Springer, London, 2002,
  pp.~281--295. \MR{1916139}

\bibitem{sawarteconjecture}
Gary~M. McGuire and A.~R. Calderbank, \emph{Proof of a conjecture of {S}arwate
  and {P}ursley regarding pairs of binary {$m$}-sequences}, IEEE Trans. Inform.
  Theory \textbf{41} (1995), no.~4, 1153--1155. \MR{1366759}

\bibitem{Niho1972MultiValuedCF}
Yoji Niho, \emph{Multi-valued cross-correlation functions between two maximal
  linear recursive sequences}, Ph.D. thesis, University of Southern California,
  Los Angeles, 1972.

\bibitem{sarwate}
D.~V. {Sarwate} and M.~B. {Pursley}, \emph{Crosscorrelation properties of
  pseudorandom and related sequences}, Proceedings of the IEEE \textbf{68}
  (1980), no.~5, 593--619.

\end{thebibliography}

\end{document}